\documentclass[a4paper,12pt]{article}
\usepackage{amssymb}
\usepackage{amsthm}
\usepackage{amsmath}
\usepackage{graphicx}
\usepackage{psfrag}
\usepackage{array,hhline}
\usepackage{enumitem}

\newtheorem{lemma}{Lemma}[section]
\newtheorem{theorem}[lemma]{Theorem}
\newtheorem{proposition}[lemma]{Proposition}
\newtheorem{conjecture}[lemma]{Conjecture}
\newtheorem{corollary}[lemma]{Corollary}
\theoremstyle{definition}
\newtheorem{definition}[lemma]{Definition}
\newtheorem{remark}[lemma]{Remark}

\numberwithin{equation}{section}
\numberwithin{figure}{section}

\newcommand{\Bset}{\mathcal{B}}
\newcommand{\Cset}{\mathcal{C}}

\newcommand{\Iset}{\mathcal{I}}
\newcommand{\Lset}{\mathcal{L}}
\newcommand{\Mset}{\mathcal{M}}

\newcommand{\Xset}{\mathcal{X}}
\newcommand{\Yset}{\mathcal{Y}}

\begin{document}

\newcommand{\RR}{\mathbb{R}^2}
\newcommand{\RN}{\mathbb{R}^n}

\title{On maximal curves of $n$-correct sets}

\author{H. A. Hakopian, G. K. Vardanyan, and N. K. Vardanyan\\ \\
 \textit{Yerevan State University, Yerevan, Armenia}\\   \textit{Institute of Mathematics of NAS RA}}
\date{}
\maketitle

\begin{abstract} 
Suppose $\Xset$ is an $n$-correct set of nodes in the plane, that is, it admits a unisolvent interpolation with bivariate polynomials of total degree less than or equal to $n.$
Then an algebraic curve $q$ of degree $k\le n$ can pass through at most $d(n,k):={{n+2}\choose {2}}-{{n+2-k}\choose {2}}$ nodes of $\Xset.$  A  curve $q$ of degree $k\le n$ is called maximal if it passes through exactly $d(n,k)$ nodes of $\Xset.$ In particular, a maximal line is a line passing through $d(n,1)=n+1$ nodes of $\Xset.$ Maximal curves are an important tool for the study of $n$-correct sets. We present new properties of maximal curves, as well as extensions of known properties.
\end{abstract}

{\bf Keywords:} Bivariate interpolation, $n$-correct set, $GC_n$-set, algebraic curve, maximal curve, maximal line.

{\bf Mathematics Subject Classification 2020:} 41A05, 41A63, 14H50

\section{Introduction}

Denote the space of bivariate polynomials by $\Pi.$ Denote the space of bivariate polynomials of total degree not exceeding $n,\ n\ge 0,$ by $\Pi_n$. We have that
$$N:=N_n:=\dim\Pi_n={{n+2}\choose {2}}.$$
It is convenient to set \begin{equation}\label{n<0}N_n=0,\ \hbox{if}\ n<0.
\end{equation}
Let
$\Xset_s=\{ (x_1, y_1), \dots , (x_s, y_s) \}$ be a set of $s$ distinct nodes in the plane.\\
The problem of finding a polynomial $p \in \Pi_n$ satisfying the conditions
\begin{equation}\label{eq1}
p(x_i, y_i) = c_i, \ \ \quad i = 1, 2, \dots s  ,
\end{equation}
for a given data $\bar c:=\{c_1, \dots, c_s\},$ is called \emph{bivariate interpolation problem.}
\begin{definition}
A set of nodes $\Xset_s$ is called
$n$-\emph{solvable} if for any data $\bar c$ there exists a
polynomial $p \in \Pi_n$ satisfying the conditions
\eqref{eq1}.
\end{definition}
\begin{definition}
A set of nodes $\Xset_s$ is called
$n$-\emph{correct} if for any data $\bar c$ there exists a
\emph{unique} polynomial $p \in \Pi_n$ satisfying the conditions
\eqref{eq1}.
\end{definition}

The conditions \eqref{eq1} give a system of $s$ linear equations with $N$ unknowns, which are the coefficients of the polynomial $p.$ The $n$-correctness means that the linear system has a unique
solution for any right side values $\bar c.$ 

Hence, the following is a necessary condition of poisedness: $s = N.$
Thus, one may consider $n$-correctness only with sets of nodes $\Xset_N.$ 
In the latter case we have the following
\begin{proposition} \label{correctii}
The $n$-correctness of the set $\Xset:=\Xset_N$ is equivalent to each of the following conditions:
\begin{enumerate}
\item
The set $\Xset$ is $n$-solvable.
\item
$p \in \Pi_n,\ p|_\Xset=0\implies p = 0.$
\end{enumerate}
\end{proposition}
Here $p|_\Xset$ denotes the restriction of $p$ on $\Xset.$
\begin{definition}
A polynomial $p \in \Pi_n$ is called $n$-\emph{fundamental polynomial}
for $A := (x_k, y_k) \in \Xset,$  if
$ p|_{\Xset\setminus\{A\}}=0\ \hbox{and}\ p(A)=1.$
\end{definition}
The above fundamental polynomial we denote by 
$p_k^\star=p_A^\star=p_{A,\Xset}^\star.$

Let us mention that sometimes we call $n$-fundamental a polynomial $p\in \Pi_n$ satisfying the conditions
$p|_{\Xset\setminus\{A\}}= 0\ \hbox{and}\ p(A)\neq 0,$
since it equals a non-zero constant times $p_A^\star.$

\begin{definition}
A set of nodes $\Xset$
is called $n$-\emph{independent} if any node $A\in \Xset$ has an $n$-fundamental polynomial $p_{A,\Xset}^\star.$ Otherwise, $\Xset$
is called $n$-\emph{dependent.}
\end{definition}
Evidently the fundamental polynomials are linearly independent. Hence, a necessary condition of $n$-independence of $\Xset$ is: $\#\Xset\le N.$

If a set of nodes $\Xset_s$
is $n$-independent then  
the following Lagrange formula gives a polynomial
$p\in\Pi_n$ satisfying the conditions \eqref{eq1}:
\begin{equation} \label{Lagrange}p(x,y)=\sum_{i=1}^sc_ip^\star_i(x,y).
\end{equation}
This yields 
\begin{proposition} A set of nodes $\Xset$
is $n$-independent if and only if it is $n$-solvable.
\end{proposition}
In the sequel we will need the following  
\begin{proposition} [\cite{HT},Lem.2.2] \label{HT1} Any $n$-independent  set of nodes $\Xset_s$ with the cardinality $s<N$ can be enlarged till an 
$n$-correct set $\Xset_N.$ 
\end{proposition}
A \emph{plane algebraic curve} of degree $n,\ n\ge 1,$ is the zero set of some non-zero bivariate polynomial of degree $n.$ To simplify notation, we shall use the same letter,  say $p$,
to denote the polynomial $p$ and the curve given by the equation $p(x,y)=0$.
In particular by $\ell$ we denote a linear
polynomial from $\Pi_1$ and the line defined by the equation
$\ell(x, y)=0.$

Let us mention that in expressions like $\Xset\setminus p,$ or ${\mathcal X \cap p},$ by the polynomial $p\in\Pi_n$ we mean its zero set.

We say that a node $A$ of an $n$-correct set $\Xset$ uses a line $\ell$ if
$$p_{A, \Xset}^\star= \ell r, \ \hbox{where} \ r\in\Pi_{n-1}.$$
Let us now consider a particular type of $n$-correct sets that satisfy a so-called geometric characterization (GC) property  (see \cite{CY77},\cite{HV}), introduced by K.C. Chung and T.H. Yao:
\begin{definition}[\cite{CY77}]
An $n$-correct set ${\mathcal X}$ is called \emph{$GC_n$-set }
 if  the
$n$-fundamental polynomial of each node $A\in{\mathcal X}$ is a
product of  linear factors.
\end{definition}
In $n$-correct sets an $n$-fundamental polynomial, as any interpolation polynomial, is unique. Hence, we readily get that in above definition the
$n$-fundamental polynomial is
product of  exactly $n$ linear factors. 

Thus each node of $GC_n$-set uses exactly $n$ lines.

\begin{definition} Let $\Xset$ be a set of nodes.
We say, that a line $\ell$ is a $k$-\emph{node line} if it passes through exactly $k$ nodes of $\Xset.$\end{definition}

The following proposition is well-known (see e.g. \cite{HJZ09b}
Prop. 1.3):
\begin{proposition}\label{prp:n+1ell}
Suppose that a polynomial $p \in
\Pi_n$ vanishes at $n+1$ points of a line $\ell.$ Then, we have that
$p|_\ell=0$ and $p = \ell  q  ,\ \text{where} \ q\in\Pi_{n-1}.$
\end{proposition}
Therefore at most $n+1$ nodes of an $n$-independent set  can be collinear. An $(n+1)$-node line $\ell$ is called a \emph{maximal line} (C. de Boor, \cite{dB07}).

Denote the set of maximal lines of an $n$-correct set $\Xset$ by $\Mset(\Xset).$
Let us bring some basic properties of maximal lines.

\begin{proposition} \label{hatk} Let $\Xset$ be an $n$-correct set. Then the following hold:
\begin{enumerate}
\item
Any two maximal lines intersect at a node of $\Xset.$
\item
Any three maximal lines are not concurrent.
\item
$\#\Mset(\Xset)\le n+2.$
\end{enumerate}
\end{proposition}

Next we present the Gasca-Maeztu (or, briefly GM) conjecture:

\begin{conjecture}[\cite{GM82},Sect.5]\label{conj:GM}
For any $GC_n$-set there exists at least
one maximal line.
\end{conjecture}

\noindent The GM conjecture is evident for $n=2.$ Till now it has been confirmed for the degrees
$n=3,4,5$ (see \cite{GM82}, \cite{B90}, \cite{HJZ14}, respectively).

For a generalization of the GM conjecture to maximal curves see \cite{HR}.

In the sequel we will make use of the following result of Carnicer and
Gasca concerning the GM conjecture:
\begin{theorem}[\cite{CG03},Thm.4.1] \label{CGth}\label{1>3} Suppose that the Gasca-Maeztu conjecture  holds for all degrees up to $n.$ Then any $GC_n$-set possesses at least three maximal lines. 
\end{theorem}
\begin{corollary} \label{CGcor} Let $\Xset$ be $GC_n$-set. Suppose that the Gasca-Maeztu conjecture  holds for all degrees up to $n.$ Then any node of $\Xset$ uses a maximal line.
\end{corollary}
Indeed, in view of Theorem \ref{1>3}, there exist three maximal lines which, according to Proposition \ref{hatk}, (ii), are not concurrent. Therefore for any node  $A\in\Xset$  there is a maximal line not passing through it. Thus, in view of Proposition \ref{prp:n+1ell}, the node $A$ uses the latter line.

\begin{definition}
Given an $n$-correct set $\Xset$. We say that a node $A \in \Xset$ uses an algebraic curve $q$ of degree $k\le n,$
if $q$ divides the fundamental polynomial:
$$p_{A, \Xset}^\star= q r, \ \hbox{where} \ r\in\Pi_{n-k}.$$
\end{definition}

\noindent Set  for $n,k\ge 0,$
\begin{equation}\label{dnk1}d(n, k) := N_n - N_{n-k}.\end{equation}
Note that for $0\le k\le n+2$ we have:
\begin{equation}\label{dnk2}d(n, k) = (n-k+2)+(n-k+3)+\cdots+(n+1)\ \left[={k(2n+3-k)\over 2}\right].\end{equation}

\noindent While for $k\ge n+1,$ in view of the relation \eqref{n<0}, we have: 
\begin{equation*}\label{k>n}d(n, k) = N_n.\end{equation*}
Now note that if $0\le k\le \min(m,n)$ then we have that
\begin{equation} \label{nmk1} d(n,m)-d(n,k)=d(n-k,m-k).\end{equation}
Indeed, we have that 
$$d(n,m)-d(n,k)=N_n - N_{n-m}-(N_n - N_{n-k})=N_{n-k}- N_{n-m}=d(n-k,m-k).$$
Then note that $m\le n$ and $0\le k\le n- m+2$ imply
\begin{equation} \label{nmk2}  d(n,m)-mk=d(n-k,m). \end{equation} 
Indeed, we have that 
$$d(n,m)-mk=(n-m+2)+(n-m+3)+\cdots+(n+1)-mk$$ 
$$= (n-m-k+2)+(n-m-k+3)+\cdots+(n-k+1)=d(n-k,m-k).$$

The following is a generalization of Proposition \ref{prp:n+1ell}:
\begin{proposition} [\cite{Raf},Prop.3.1]\label{maxcurve}
Let $\Xset$ be an $n$-correct set and $q$ be an algebraic curve of degree $k \le n$ with no multiple components. Then:
\begin{enumerate}
\item
Any subset of $q$ containing more than $d(n,k)$ nodes of $\Xset$ is
$n$-dependent.
\item
Any subset of $q$ containing exactly $d(n,k)$ nodes of $\Xset$ is $n$-independent if and only if
$$\quad p\in {\Pi_{n}}\ \text{and}\ \ p|_{{\mathcal X}} = 0 \implies  p = qr,\ \hbox{where}\ r \in \Pi_{n-k}.$$
\end{enumerate}
\end{proposition}
For lines we have that $d(n, 1) = n+1$ and any $n+1$ points in a line are $n$-independent.

For conics (algebraic curves of degree $2$) we have that $d(n, 2) = 2n+1$. It is well-known that a set $\Xset$ of $\le 2n+1$ points is $n$-independent
if and only if it has no $n+2$ collinear points (see \cite{Hak00} for the case of multiplicities). If $\Xset$ is a subset of an irreducible conic then at most $2$ points of $\Xset$ can lie on the same line. 
Thus any set of $\le 2n+1$ points
located on an irreducible conic is $n$-independent. 

For cubics (algebraic curves of degree $3$) and curves of higher degrees things are more complicated. In particular not any set of $d(n, 3) = 3n$ nodes located in a cubic is $n$-independent (see \cite{HM12}).
  
In the sequel we will need the following 
\begin{proposition}[\cite{HT},Prop.3.5] \label{HT2} Let $q$ be an algebraic curve of degree $k \le n$ with no multiple components, and $\Xset_s\subset q$ be an $n$-independent  set  of cardinality $s,\ s<d(n,k).$ Then $\Xset_s$ can be enlarged till a maximal $n$-independent set  $\Xset_d\subset q,$ of cardinality $d=d(n,k).$ 
\end{proposition}

\section{Maximal curves}

\noindent Proposition \ref{maxcurve} implies that at most $d(n,k)$ $n$-independent nodes lie in any curve $q$ of degree $k \le n$. This motivates the following
\begin{definition}[\cite{Raf},Def.3.1]\label{def:maximal}
Let $\Xset$ be an $n$-correct set. A curve $f$ of degree $k \le n$ passing through  $d(n,k)$ nodes
of $\mathcal X$ is called \emph{maximal curve.}
\end{definition}
Since $d(n,n)=N-1$ we get that each fundamental polynomial of $\Xset$ is a maximal curve of degree $n.$

\noindent Note that evidently the curve $f$ of Definition \ref{def:maximal} has no multiple components.

The following proposition gives a characterization of the maximal curves:
\begin{proposition}[\cite{Raf},Prop.3.3] \label{maxcor}
Let $\Xset$ be an $n$-correct set. Then
a curve $f\in\Pi$ of degree $k\le n$ is a maximal curve for $\Xset$ if and only if 
it is used by any node of the set $\Xset\setminus f.$
\end{proposition} 

Note that this, in view of the Lagrange formula \eqref{Lagrange}, follows directly from Proposition \ref{maxcurve}. 

Proposition \ref{maxcor} and the subsequent Corollary \ref{maxtim}, due to L. Rafayelyan \cite{Raf}, reveal the primary significance of maximal curves in the theory of bivariate polynomial interpolation. 

\begin{corollary}[\cite{Raf},Prop.3.4] \label{maxtim}  Let $\Xset$ be an $n$-correct set. Then: 
\begin{equation}\label{maxmax}
f\in\Pi_k \hbox{ is a maximal curve}  \iff \Xset\setminus f  \hbox{ is an $(n-k)$-correct set.}\end{equation} 
\end{corollary}
For the sake of completeness we present
\begin{proof} 
 $(\Leftarrow)$ 
Denote by $\Bset:=\Xset\setminus f.$ We have that
\begin{equation}\label{nnk}N_{n-k}=\#\Bset=N_n-\#(\Xset\cap f).\end{equation}
From here we obtain that $\#(\Xset\cap f)=N_n-N_{n-k}=d(n,k),$ i.e., $f$ is a maximal curve of degree $k.$

$(\Rightarrow)$
In view of  Proposition \ref{maxcor}  for any $A\in \Xset\setminus f$ we have that
$$p^\star_{A,\Xset} = q_Af,\ \hbox{where}\ q\in\Pi_{n-k}.$$
This implies that
$q_A= p^\star_{A,\Bset}.$ On the other hand we have that $\#\Bset=N-d(n,k)=N_{n-k}.$ Therefore, in view of Proposition \ref{correctii}(i), the set $\Bset$ is $(n-k)$-correct.
\end{proof} 

Now, we obtain  

\begin{corollary}\label{maxhq} Let $\Xset$ be an $n$-correct set, $f$ and $fh$ be
maximal curves, where   $\deg f=k$ and $\deg h=m.$
Then 
the curve $h$ is a maximal curve for the $(n-k)$-correct set $\Xset\setminus f.$
\end{corollary}
\begin{proof}
For any $A\in \Xset\setminus (fh)$ we have that 
\begin{equation} \label{hgqa} p_{A,\Xset}^\star = fhr_A,\end{equation} 
where $r_A\in\Pi_{n-m-k}.$

Denote $\Yset:=\Xset\setminus f.$ 
Then, from \eqref{hgqa} we readily obtain that
$$p_{A,\Yset}^\star = hr_A\ \forall A\in\Yset\setminus h.$$
This, in view of Proposition \ref{maxcor}, completes the proof.
\end{proof}

\begin{corollary}\label{maxtim2} Let $\Xset$ be an $n$-correct set, $f_1$ and $f_2$ be
maximal curves, with no common components. Suppose that  $\deg f_i=k_i,\ i=1,2,$ where $k_1+k_2\le n.$ Then: 
\begin{enumerate}
\item
The curve $f_1f_2$ is a maximal curve of degree $k_1+k_2.$
\item
The curve $f_2$ is a maximal curve for the $(n-k_1)$-correct set $\Xset\setminus f_1.$
\end{enumerate}
\end{corollary}
\begin{proof}
Note first that $\deg f_2 =k_2\le n-k_1.$ Then any node $A\in \Xset\setminus (f_1\cup f_2)$ uses both curves $f_1$ and $f_2.$ Since the latter curves have no common components we obtain that
\begin{equation}\label{astgm}p_{A,\Xset}^\star = f_1 f_2 r, \ \hbox{where}\ r\in\Pi_{n-k_1-k_2}.\end{equation}
This, in view of  Proposition \ref{maxcor}, implies (i).

Now (ii) readily follows from Proposition \ref{maxhq}.
\end{proof} Let us mention that Corollary \ref{maxtim2} (ii) was proved in (\cite{Raf},Prop.3.4).
\begin{remark} \label{rem} In view of the relation \eqref{astgm} one readily gets that if  $f_1$ and $f_2$ are
maximal curves, with no common components, and  $\deg f_i=k_i,\ i=1,2,$ where $k_1+k_2\ge n+1,$ then $\Xset\subset f_1\cup f_2.$
\end{remark}
 \section{Intersection of two maximal curves}
Denote by $\Iset(f_1,f_2)$ the set of intersection points of the curves $f_1$ and $f_2.$
Denote also $\Iset_\Xset(f_1,f_2):=\Iset(f_1,f_2)\cap \Xset.$ 
\begin{proposition}[\cite{Raf},Prop.3.4]\label{prop5}
Let $\Xset$ be an $n$-correct set, $f_1$ and $f_2$ be
maximal curves, with no common components. Suppose that  $\deg f_i=k_i,\ i=1,2,$ where 
\begin{equation} \label{k1k2+} k_1+k_2\le n.\end{equation}  
Then the curves $f_1$ and $f_2$ intersect at exactly $k_1k_2$ distinct points, which all are nodes of $\Xset:$
$$\#\Iset_\Xset(f_1,f_2) =k_1k_2.$$
\end{proposition}
Let us present a short
\begin{proof} In view of Corollary \ref{maxtim2}(i), and the relations \eqref{nmk1}, \eqref{nmk2}, we have that

\noindent $\#\Iset_\Xset(f_1,f_2)=\#\left\{f_1\cap \Xset\right\}+\#\left\{f_2\cap \Xset\right\}-\#\left\{f_1f_2\cap \Xset\right\} = d(n,k_1)+d(n,k_2)-d(n,k_1+k_2)=d(n,k_1)-d(n-k_2,k_1)= k_1 k_2.$
\end{proof}

\begin{figure}
\begin{center}
\includegraphics[width=4.0cm,height=4.cm]{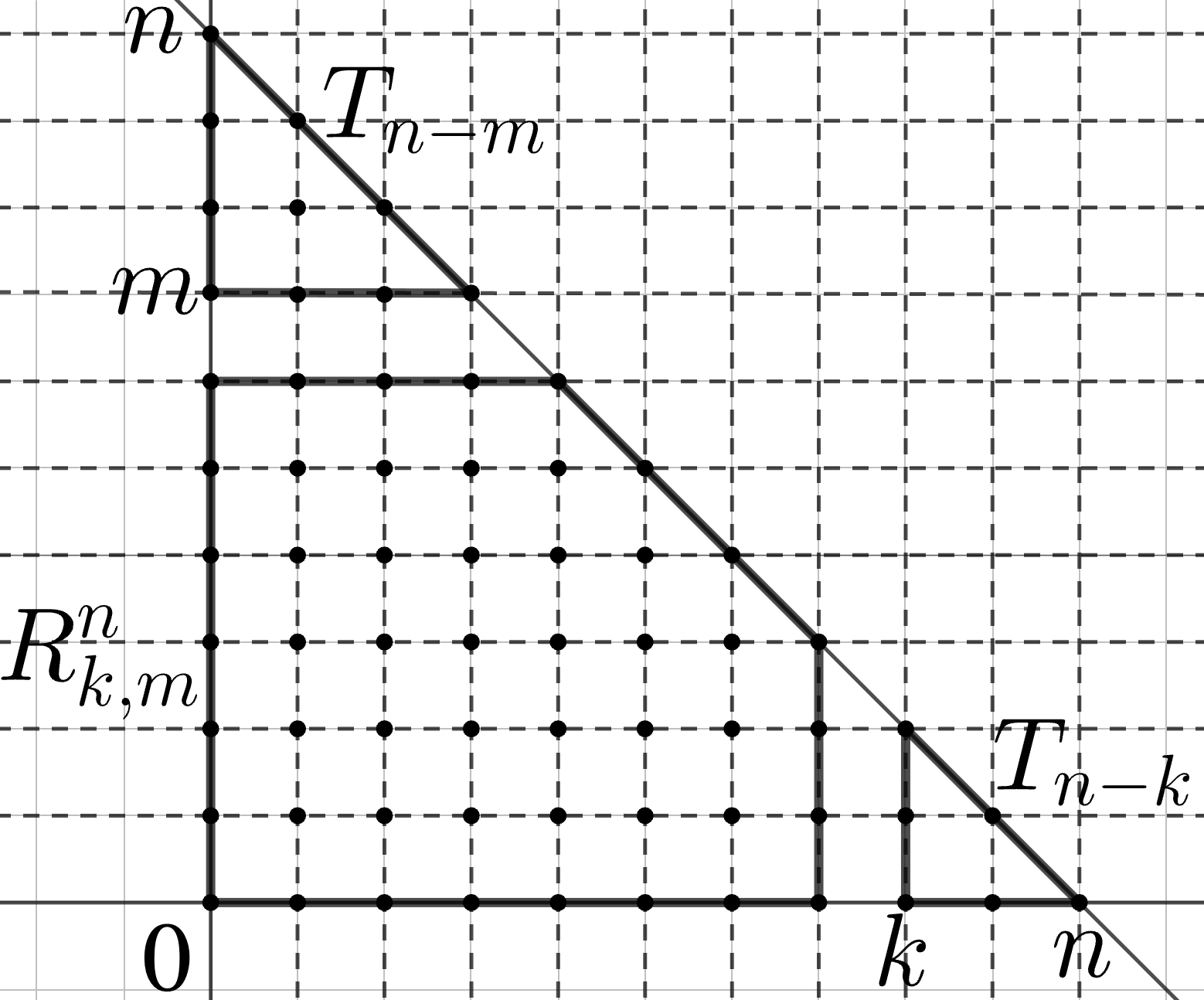}
\end{center}
\caption{The set $R_{k,m}^n$ for $k+m\ge n+3.$} \label{Fig1}
\end{figure}

Denote the following triangular lattice of dimension $n$ by
$$T_n:=T^{i_0,j_0}_n=\left\{(i+i_0,j+j_0) : i,j\ge 0,\ i+j\le n\right\},$$ 
where $i_0,j_0\ge 0.$

Note that 
$$N_n=\#T_n.$$ 
Denote also the following
rectangular lattice of dimension $k\times m$ by
$$ R_{k,m}:=\left\{(i,j) : 0\le i\le k-1, 0\le j\le m-1\right\}.$$ 
Now, set (see Fig.\ref{Fig1})
$$ R_{k,m}^n:=R_{k,m} \cap T^{0,0}_n=\left\{(i,j) : 0\le i\le k-1, 0\le j\le m-1, i+j\le n\right\}.$$
Then, denote
$$H_{k,m}^n:=\#R_{k,m}^n.$$
Note that if $k + m\le n+2$ then  $R_{k,m}\subset T^{0,0}_n$ (see Fig.\ref{Fig2}) and hence
$$ R_{k,m}^n= R_{k,m}.$$
Therefore we obtain that
\begin{equation}\label{zs0}H_{k,m}^n =km \ \hbox{if}\ k + m\le n+2.\end{equation}
Next, let us verify that 
\begin{equation}\label{zsss} 
H_{k,m}^n = N_n-N_{n-k}-N_{n-m} \ \hbox{if}\  k + m\ge n+1.
\end{equation}
Indeed, this readily follows from  the following disjoint representation of  $T^{0,0}_n$ (see Fig.\ref{Fig1}):
$$T^{0,0}_n=H_{k,m}^n\cup T^{k,0}_{n-k}\cup T^{0,m}_{n-m}.$$ 

\begin{figure}
\begin{center}
\includegraphics[width=4.0cm,height=4.cm]{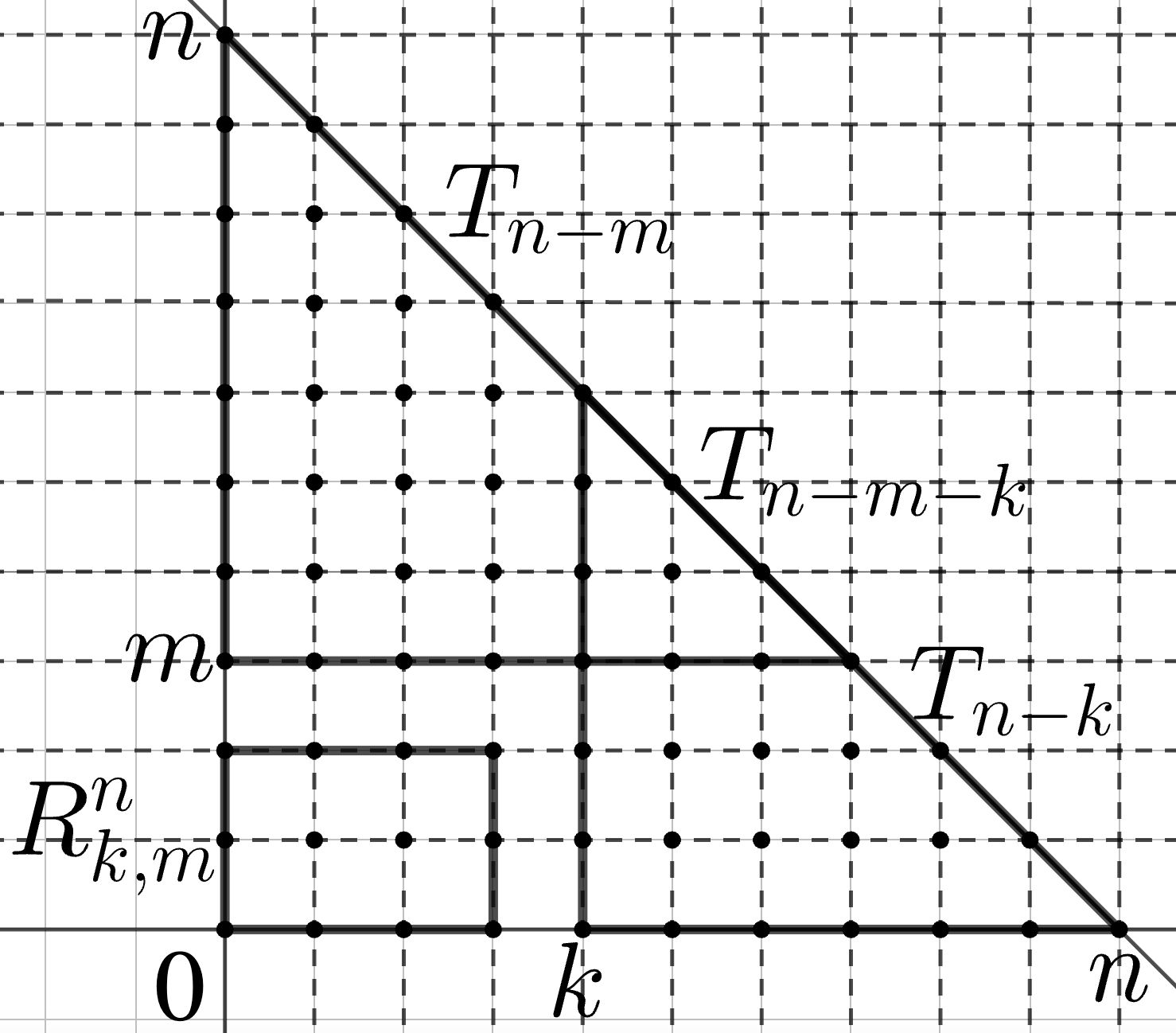}
\end{center}
\caption{The set $R_{k,m}^n$ for $k+m\le n+2.$} \label{Fig2}
\label{Fig2}
\end{figure}

Then, consider the following simple relation: 
$$(k-1,m-1)\notin R_{k,m}^n \ \iff \ k + m\ge n+3.$$
This implies that
\begin{equation}\label{zs} H_{k,m}^n \le km-1 \ \iff \ k + m\ge n+3.
\end{equation}

Note, that in the case $k + m\le n,$ the following relation takes place:
\begin{equation}\label{zsss2} H_{k,m}^n = N_n-N_{n-k}-N_{n-m}+N_{n-k-m}. 
\end{equation}
Indeed, it follows from the following equality (see Fig.\ref{Fig2})
$$\#H_{k,m}^n=\#T^{0,0}_n-\# T^{k,0}_{n-k}-\#T^{0,m}_{n-m}+\#T^{k,m}_{n-k-m}.$$
Let us mention that the relation \eqref{zsss2} holds in all cases. In particular the relation \eqref{zsss} is its special case, since $N_s=0,$ if $s<0.$

Then, suppose that $f_1$ and $f_2$ are algebraic curves of degrees $k_1$ and $k_2,$ respectively, intersecting at exactly $k_1k_2$ distinct points:
$\#\Iset(f_1,f_2)=k_1k_2.$ 
Now, according to the Cayley-Bacharach theorem, the cardinality of maximal $n$-independent subsets of $\Iset(f_1,f_2)$ equals to $H_{k_1,k_2}^n.$   The latter quantity is called the \emph{Hilbert function of the set $\Iset(f_1,f_2)$ for the degree $n.$} 

Thus, the above mentioned curves $f_1$ and $f_2$ may have at most
$H_{k_1,k_2}^n$ intersection points inside an $n$-correct set $\Xset.$

Next, suppose that $ k_1 + k_2\ge n+3.$ Then, in view of  \eqref{zs}, no curves $f_1$ and $f_2,$ of degrees $k_1$ and $k_2,$ respectively, can have $k_1k_2$ intersection points inside an $n$-correct set $\Xset.$

Next proposition, which extends Proposition \ref{prop5}, asserts that for  maximal curves the above mentioned extremal case takes place.

\begin{proposition}\label{prop6}
Let $\Xset$ be an $n$-correct set, $f_1$ and $f_2$ be
maximal curves, with no common components. Suppose that  $\deg f_i=k_i,\ i=1,2.$ Then the curves $f_1$ and $f_2$ intersect at exactly $H_{k_1,k_2}^n$ nodes of $\Xset:$
$$\#\Iset_\Xset(f_1,f_2) =H_{k_1,k_2}^n.$$
\end{proposition}
\begin{proof} If $k_1+k_2\le n$ then, in view of \eqref{zs0}, the statement follows from Proposition \ref{prop5}.
  
Now, suppose that $k_1+k_2 \ge n+1$. Then, in view of Remark \ref{rem} we have that 
\begin{equation*} 
	\Xset \subset f_1 \cup f_2.
\end{equation*}
This implies that
\begin{equation*}\#\Iset_\Xset(f_1,f_2) = \#\left\{f_1\cap \Xset\right\}+\#\left\{f_2\cap \Xset\right\}-\#\left\{\Xset\right\} \end{equation*}
$$= d(n,k_1)+d(n,k_2)- N= N-\left(N-d(n,k_1)\right)-(N-d(n,k_2))$$
$$=N-N_{n-k_1}- N_{n-k_2}=H_{k_1,k_2}^n.$$
The last equality here follows from \eqref{zsss}.
\end{proof}

\begin{remark} \label{rem1} Note that, in view of Proposition \ref{prop6} and the relation \eqref{zs0}, we obtain that Proposition \ref{prop5} holds under a weaker than \eqref{k1k2+} restriction 
$k_1+k_2\le n+2$ (see \cite{Raf},Prop.3.4).  Also note that, as we mentioned earlier, Proposition \ref{prop5} is not valid if $k_1+k_2\ge n+3.$\end{remark} 

\begin{remark} Note that $\Xset_0:=T^{0,0}_n$ is an $n$-correct set. Consider the curves $f(x,y)=x(x-1)\cdots (x-k+1)$ and $g(x,y)=y(y-1)\cdots (y-m+1),$ which are maximal for 
$\Xset_0$ of degrees $k$ and $m,$ respectively. Then we have that (see the cases of Fig. \ref{Fig1} and \ref{Fig2})
$$f \cap g\cap \Xset_0=R_{k,m}^n.$$
\end{remark}

Below we consider the intersection of two maximal curves which have a common component. Next proposition asserts, in particular, that the greatest common divisor of such maximal curves itself is a maximal curve.

\begin{proposition}\label{prop7}
Let $\Xset$ be an $n$-correct set, $f_1=hg_1$ and $f_2=hg_2$ be any 
maximal curves, whose greatest common divisor is $h.$ Suppose  that  $\deg h=m$ and $\deg g_i=s_i,\ i=1,2,$ where 
\begin{equation} \label{skm} s :=s_1+s_2+m\le n.
\end{equation}
Then the following statements hold.
\begin{enumerate}
\item
The curve $g_1g_2h$ is a maximal curve of degree $s.$
\item
The curve $h$ is a maximal curve of degree $m.$
\item
The curves $g_1$ and $g_2$ are maximal curves of degrees $s_1$ and $s_2,$ respectively, for the $(n-m)$-correct set $\Xset\setminus h.$
\item
$\#\Iset_\Xset(g_1,g_2) =s_1 s_2$ and $h\cap g_1\cap g_2\cap \Xset=\emptyset.$
\item
The curves $f_1$ and $f_2$ intersect at exactly $d(n,m)+s_1 s_2$ nodes of $\Xset:$
\begin{equation}\label{k_1k_2}\#\Iset_\Xset(f_1,f_2) =d(n,m)+s_1 s_2.\end{equation}
\end{enumerate}
\end{proposition}
\begin{proof} Suppose that $A\in\Xset\setminus g_1g_2h.$ Then we have $f_1(A)f_2(A)\neq 0.$ Hence both $f_1$ and $f_2$ divide $P^\star_{A,\Xset}.$
This readily implies that $g_1g_2h$ divides $P^\star_{A,\Xset}.$ Therefore, in view of Proposition \ref{maxcor}, we get (i).

Next, we have that
$$\#\Iset_\Xset (f_1,f_2)=\#\Iset_\Xset (hg_1,hg_2)= \#(\Xset\cap h)+\#\Iset_{\Xset} (g_1,g_2) -\#(h\cap g_1\cap g_2\cap \Xset))$$ $$\le \#(\Xset\cap h)+\#\Iset_{\Xset}(g_1,g_2)\le d(n,m)+\#\Iset_{\Xset}(g_1,g_2)\le d(n,m)+s_1 s_2.$$
Let us mention that in the last inequality we used the Bezout theorem.

Now note that if the equality cases hold in the above three inequalities, or, in other words, if \eqref{k_1k_2} holds, then we easily conclude that $\#(h\cap g_1\cap g_2\cap \Xset)=0$ (the first inequality), $h$ is a maximal curve of degree $m$ (the second inequality), and $\#\Iset_\Xset(g_1,g_2) =s_1 s_2$ (the third inequality). Thus, (iv) and (ii) will be proved, if \eqref{k_1k_2} holds.

Furthermore,  by applying Corollary \ref{maxhq}, we would obtain (iii) too, since $hg_1=f_1$ and $hg_2=f_2$ are maximal curves. 

\noindent Therefore it suffices to prove only the equality \eqref{k_1k_2}, i.e., statement (v).

Recall that the following inequality was established above:
$$\#\Iset_\Xset(f_1,f_2) \le d(n,m)+s_1 s_2.$$

Finally suppose, by way of contradiction, that the equality \eqref{k_1k_2}, does not hold. Then in view of the above inequality we obtain that
\begin{equation} \label{nm+}\#\Iset_\Xset(f_1,f_2) <d(n,m)+s_1 s_2.\end{equation}

Next, in view of Corollary  \ref{maxtim}, we have that $\Yset:=\Xset\setminus f_1$ is an $(n-m-s_1)$-correct set.
The curve $f_2$ is maximal curve of degree $m+s_2$ and passes through
$d(n,m+s_2)$ nodes of $\Xset.$ Thus $f_2$ passes through $d(n,s_2+m)-\Iset_\Xset(f_1,f_2)$ nodes of $\Yset.$ On the other hand, in view of \eqref{nm+},  \eqref{nmk1} and \eqref{nmk2}, we have that 
$$d(n,s_2+m)-\Iset_\Xset(f_1,f_2)>d(n,m+s_2)-d(n,m)-s_1 s_2=d(n-m,s_2)-s_1 s_2$$
$$=d(n-m-s_1,s_2)$$ Hence, $f_2$ passes through more than $d(n-m-s_1,s_2)$
nodes of $\Yset.$

Next we have that $h(A)\neq 0\ \ \forall A\in \Yset.$ Therefore, by taking into account the equality $f_2=h g_2,$ we obtain that the curve $g_2$ of degree $s_2$ passes through more than $d(n-m-s_1,s_2)$ nodes of $\Yset$ too. In view of \eqref{skm} this is contradiction.
Thus (v) is proved. \end{proof}

Now, we present an extension of Proposition \ref{prop7}, where the restriction \eqref{skm} is removed. Note that in that case the statement (i) cannot be proved.
\begin{proposition}\label{prop8}
Let $\Xset$ be an $n$-correct set, $f_1=hg_1$ and $f_2=hg_2$ be any 
maximal curves, whose greatest common divisor is $h.$ Suppose  that  $\deg h=m,\hbox{and}\ \deg g_i=s_i,\ i=1,2.$ Then the 
curves $f_1$ and $f_2$ intersect at exactly 
\begin{equation}\label{nm++}\#\Iset_\Xset(f_1,f_2) =d(n,m)+H_{s_1,s_2}^{n-m}\end{equation} nodes of $\Xset.$
\end{proposition}
\begin{proof} 
The case $s_1+s_2+m\le n$ is considered in Proposition \ref{prop7}. Now suppose that 
\begin{equation} \label{k1k2m} s_1+s_2+m\ge n+1.
\end{equation}

First let us verify that 
\begin{equation}\label{n+1}
	\Xset \subset g_1 \cup g_2 \cup h = f_1\cup f_2.
\end{equation} Indeed, assume conversely that $A\in \Xset\setminus g_1g_2h.$
In the proof of Proposition \ref{prop7} we verified that then $g_1g_2h$ divides $P^\star_{A,\Xset},$ which, in view of \eqref{k1k2m}, is contradiction.
  
Now, in view of \eqref{n+1}, \eqref{nmk1}, and \eqref{dnk1}, we obtain that
$$\#\Iset_\Xset(f_1,f_2) = \#\left\{f_1\cap \Xset\right\}+\#\left\{f_2\cap \Xset\right\}-\#\left\{\Xset\right\}$$ $$= d(n,m+s_1)+d(n,m+s_2)- N
= d(n,m)+d(n-m,s_1)+d(n,m+s_2)- N$$ $$= d(n,m)+d(n-m,s_1)- N_{n-m-s_2}
= d(n,m)+N_{n-m}-N_{n-m-s_1}- N_{n-m-s_2}$$ $$= d(n,m)+H_{s_1,s_2}^{n-m}.$$
In the last equality above we used the relation \eqref{zsss}. 
\end{proof}

\begin{remark} \label{rem2} Note that, in view of equality \eqref{nm++} and the relation \eqref{zs0}, we obtain that the statement (v) of Proposition \ref{prop7}, holds under a weaker than \eqref{skm} restriction: 
$$s_1+s_2+m\le n+2.$$ 
Furthermore, given the proof of Proposition \ref{prop7}, we obtain that statements (ii-iv) of Proposition \ref{prop7} are also valid under the same weaker restriction.\end{remark} 
\section{Intersection of three maximal curves}
\begin{proposition}[\cite{Raf},Prop.3.4] \label{prop10}
Let $\Xset$ be an $n$-correct set. Let also $f_1, f_2$ and $f_3$ be
three maximal curves, such that every two of them have no common components.  
Suppose also that $deg f_i=k_i, \ i=1,2,3,$ and $k_1+k_2+k_3\le n+2.$ Then we have that
$f_1\cap f_2\cap f_3=\emptyset.$
\end{proposition}
Let us present a simpler 
\begin{proof}
In view of Corollary  \ref{maxtim} we have that
        	$\Yset_1 := \Xset \setminus f_1$ is an $(n-k_1)$-correct set. Next, according to Corollary \ref{maxhq}, we have that
        	$f_2$ and $f_3$ are maximal curves in $Y_1,$ and $k_2+k_3\le n-k_1+2.$  Then, in view of Proposition \ref{prop5} and Remark \ref{rem1}, we have that $f_2\cap f_3\subset \Yset_1= \Xset \setminus f_1.$ This implies that $f_1\cap f_2\cap f_3=\emptyset.$
\end{proof}
In the next two propositions we find $\#(f_1\cap f_2\cap f_3 \cap \Xset)$ in the case of 
$$k_1+k_2+k_3> n+2.$$
\begin{proposition} \label{prop1011}
Let the set $\Xset$ and the curves $f_1, f_2$ and $f_3,$ be
as in Proposition \ref{prop10}. Let also $k_1+k_2+k_3\ge n+1.$
Then we have the following three expressions for $\#(f_1\cap f_2\cap f_3 \cap \Xset):$
\begin{enumerate}
\item
$N-d(n,k_1)-d(n,k_2)-d(n,k_3)+H_{k_1,k_2}^n+H_{k_2,k_3}^n+H_{k_1,k_3}^n,$
\item
$N-N_{n-k_1}-N_{n-k_2}-N_{n-k_3}+N_{n-k_1-k_2}+N_{n-k_2-k_3}+N_{n-k_1-k_3},$
\item
$N-d(n-k_1, k_2)-d(n-k_2,k_3)-d(n-k_3,k_1).$
\end{enumerate}
\end{proposition}
\begin{proof}
First let us verify that 
\begin{equation}\label{n+1+}
	\Xset\subset f_1\cup f_2 \cup f_3. 
\end{equation} Indeed, assume conversely that $A\in \Xset\setminus f_1f_2f_3.$
Then $f_1f_2f_3$ divides $P^\star_{A,\Xset},$ which, in view of the condition $k_1+k_2+k_3\ge n+1,$ is contradiction.

Now denote $\gamma=\#(f_1\cap f_2\cap f_3 \cap \Xset).$ 
In view of \eqref{n+1+} we have that
$$ N=\#[(f_1\cup f_2\cup f_3 \cup)\cap \Xset]=   \#(f_1 \cap \Xset)+ \#(f_2\cap \Xset)+ \#(f_3 \cap \Xset)$$ $$-
\#(f_1\cap f_2 \cap \Xset) - \#(f_2\cap f_3 \cap \Xset) - \#(f_1\cap f_3 \cap \Xset) +\gamma$$
$$=d(n,k_1)+d(n,k_2)+d(n,k_3)-H_{k_1,k_2}^n-H_{k_2,k_3}^n-H_{k_1,k_3}^n+\gamma.$$ This proves the statement (i).

Then we continue by using the equality \eqref{zsss2}:
$$N= d(n,k_1)+d(n,k_2)+d(n,k_3)-3N+2N_{n-k_1}+2N_{n-k_2}+2N_{n-k_3}$$ $$-N_{n-k_1-k_2}-N_{n-k_2-k_3}-N_{n-k_1-k_3}+\gamma.$$
Next, in view of the definition of $d(n,k),$ we have that
$$N= (N-N_{n-k_1})+(N-N_{n-k_2})+(N-N_{n-k_3})-3N+2N_{n-k_1}+2N_{n-k_2}+2N_{n-k_3}$$
$$-N_{n-k_1-k_2}-N_{n-k_2-k_3}-N_{n-k_1-k_3}+\gamma$$ 
$$=N_{n-k_1}-N_{n-k_1-k_2}+N_{n-k_2}-N_{n-k_2-k_3}+N_{n-k_3}-N_{n-k_1-k_3}+\gamma$$
$$=d(n-k_1, k_2)+d(n-k_2,k_3)+d(n-k_3,k_1)+\gamma.$$
In the last equality, which proves the statement (iii), we again used the definition of $d(n,k).$ Note that the previous equality proves the statement (ii).
\end{proof}
The following proposition gives a clearer expression for $\#(f_1\cap f_2\cap f_3 \cap \Xset).$

\noindent Denote
$\sigma:=k_1+k_2+k_3-(n+2).$ 
\begin{proposition} \label{prop2022}
Let the set $\Xset$ and the curves $f_1, f_2$ and $f_3,$ be
as in Proposition \ref{prop10}. Let also $k_1+k_2+k_3\ge n+1,$ i.e., $\sigma\ge -1.$
Then we have the following two expressions for $\#(f_1\cap f_2\cap f_3 \cap \Xset):$
\begin{enumerate}
\item
${1\over 2} \sigma(\sigma+1)-N_{k_1+k_2-n-3}-N_{k_2+k_3-n-3}-N_{k_3+k_1-n-3},$
\item
${1\over 2} \sigma(\sigma+1),\ \hbox{provided that}\  k_i+k_j\le n+2\ \forall\ 1\le i<j\le 3.$ 
\end{enumerate}
\end{proposition}
\begin{proof} 
Let us first prove (ii).
From Proposition \ref{prop1011}(iii), in view of  \eqref{dnk2} and the condition in (ii), we obtain that
\begin{equation}\label{112233}\#(f_1\cap f_2\cap f_3 \cap \Xset)\end{equation}
$$= N- {1\over 2} k_2(2n-2k_1-k_2+3)- {1\over 2} k_3(2n-2k_2-k_3+3)- {1\over 2} k_1(2n-2k_3-k_1+3)$$
$$= N-{1\over 2} [(k_1+k_2+k_3)(2n+3)]-{1\over 2}[k_1^2+k_2^2+k_3^2+2k_1k_2+2k_2k_3+2k_3k_1]$$ 
$$= N-{1\over 2} [(k_1+k_2+k_3)(2n+3)]-{1\over 2}[k_1+k_2+k_3]^2$$ 
$$= N-{1\over 2} [(k_1+k_2+k_3)(2n+3-k_1-k_2-k_3)]$$ 
\begin{equation}\label{aabbcc}= {1\over 2} [(n+2)(n+1)]-{1\over 2} [(n+2+\sigma)(n+1-\sigma)]= {1\over 2} \sigma(\sigma+1).\end{equation}

Now, we are going to prove (i), which is an extension of (ii).
For this end let us define the following quantity $\widetilde d(n,k)$ for all $n,k\ge 0:$ 
\begin{equation}\label{dnk3+}\widetilde d(n, k) := (n-k+2)+(n-k+3)+\cdots+(n+1)\ \left[={k(2n+3-k)\over 2}\right].\end{equation}

Then, let us verify that 
\begin{equation}\label{dnk3++}\widetilde d(n, k)=d(n,k)-N_{-n+k-3},\ \forall n,k\ge 0.\end{equation}
Indeed, if $0\le k\le n+2$ then, in view of \eqref{dnk2}, we have that $ \widetilde d(n, k)=d(n,k)=d(n,k)-N_{-n+k-3}.$
While if  $k\ge n+3$ then we have that $ \widetilde d(n, k)=N_n-N_{-n+k-3}=d(n,k)-N_{-n+k-3}.$

Next, observe that the proof of statement (ii), more precisely the transformations that begin with the line after \eqref{112233} and end with the line \eqref{aabbcc}, in view of \eqref{dnk3+}, indicate that
\begin{equation}\label{dnk3-}N-\widetilde d(n-k_1, k_2)-\widetilde d(n-k_2,k_3)-\widetilde d(n-k_3,k_1)={1\over 2} \sigma(\sigma+1),\end{equation}
where $0\le k_i\le n,\ i=1,2,3.$

\noindent After this preparatory work, let us prove the statement (i). By using Proposition \ref{prop1011}(iii), and the 
relation \eqref{dnk3++}, we obtain that 
$$\#(f_1\cap f_2\cap f_3 \cap \Xset)=N-d(n-k_1, k_2)- d(n-k_2,k_3)-d(n-k_3,k_1)$$ $$=N-\widetilde d(n-k_1, k_2)-\widetilde d(n-k_2,k_3)-\widetilde d(n-k_3,k_1)$$ 
$$-N_{k_1+k_2-n-3}-N_{k_2+k_3-n-3}-N_{k_3+k_1-n-3}$$
Now, in view of the relation \eqref{dnk3-}, we get
$$={1\over 2} \sigma(\sigma+1)-N_{k_1+k_2-n-3}-N_{k_2+k_3-n-3}-N_{k_3+k_1-n-3}.$$
\end{proof}

\section{The maximal curves in $GC_n$-sets}

\begin{proposition} \label{nor} Let $\Xset$ be $GC_n$-set and $f$ be a maximal curve of degree $k,\ 1\le k\le n.$ Then $f$ is a product of $k$ distinct lines. 

Furthermore, if the Gasca-Maeztu conjecture  holds for all degrees up to $n,$ then at least one of the above $k$ lines is a maximal line.
\end{proposition}
\begin{proof} Consider a node $A\in\Xset\setminus f.$ According to Proposition \ref{maxcor} we have that
$$p^\star_A=f q, \ \hbox{where}\ q\in \Pi_{n-k}.$$
On the other hand, since $\Xset$ is $GC_n$-set, we have that
\begin{equation}\label{n1nk0}p^\star_A=\ell_1\cdots\ell_n, \ \hbox{where}\ \ell_i\in \Pi_1.\end{equation}
From these two equalities, we obtain readily that $f$ is a product of $k$ lines from the right hand side of \eqref{n1nk0}, say,
\begin{equation}\label{n1nk}f=\ell_{1}\cdots\ell_{k}.  
\end{equation}
Now, suppose that the Gasca-Maeztu conjecture  holds true. Let us show that at least one of the lines on the right-hand side of \eqref{n1nk} is a maximal line.
Assume, by way of contradiction, that $\ell_{i}$ is not a maximal line  $\forall \ 1\le i\le k.$

Consider a maximal line $\alpha_1.$ We have that it is not a component of $f.$
Therefore, in view of Proposition \ref{prop5}, the line $\alpha_1$ intersects the maximal curve $f$ in exactly $k$ nodes. Thus it intersects each line $\ell_{i},\ 1\le i\le k$ at a node.

Now, according to Corollary \ref{maxtim2}, we have that the curve $f$ is a maximal also in the $GC_{n-1}$ set $\Xset_{1}:=\Xset\setminus \alpha_1.$ 

Then, consider a maximal $n$-node line $\alpha_2$ in $\Xset_{1}.$ We have that it is not a component of $f.$ Indeed, assume, by way of contradiction, that this line is a component of $f,$ say $\alpha_2=\ell_1,$ then evidently the line $\ell_{1}$ is an $(n+1)$ node line in $\Xset,$ i.e., it is a maximal line there.  Thus $\alpha_2$ also intersects each line $\ell_{i},\ 1\le i\le k$ at a node of $\Yset_{1}.$  Next set $\Xset_{2}:=\Xset\setminus (\alpha_1\cup \alpha_2),$ and so on.

By continuing this way we get maximal lines $\alpha_i$ in 
$$\Xset_{i}:=\Xset\setminus (\alpha_1\cup\cdots\cup \alpha_{i-1}),\ i=1,\ldots, n-k,$$ and the curve $f,\ \deg f=k,$ is a maximal curve also in $GC_k$-set $\Xset_{n-k}.$

Therefore $f$ is a fundamental polynomial of a node  $B\in\Xset_{n-k}.$ Hence by Corollary \ref{CGcor} we get that some component of $f,$ say $\ell_{1},$ is a maximal line in $\Xset_{n-k}.$ But then evidently the line $\ell_{1}$ is an $(n+1)$ node line in $\Xset,$ i.e., it is a maximal line there, which is contradiction. 
\end{proof}
Next, we characterize maximal curves in $GC_n$-sets.
\begin{proposition} \label{nor2} Let $\Xset$ be $GC_n$-set. Suppose that the Gasca-Maeztu conjecture  holds for all degrees up to $n.$ Then  $f$ is a maximal curve of degree $k,\ 1\le k\le n,$ if and only if
\begin{equation}\label{mmm} f=\ell_1\cdots\ell_k,
\end{equation}
where the set $\ell_i\setminus (\ell_1\cup\cdots\cup\ell_{i-1})$ contains exactly $n+2-i$ nodes of $\Xset,\ i=1,\ldots, k.$ 
\end{proposition}
\begin{proof} $(\Rightarrow)$ Let us use induction by $k.$ The case $k=1$ is trivial.
Assume that $f$ is a maximal curve of degree $k\ge 2.$ According to Proposition \ref{nor} $f$ is a product of $k$ lines, one of which is a maximal line. Denote the latter line by $\ell_1.$ Thus we have that
$$f=\ell_1 g,\ g\in \Pi_{k-1}.$$
Now, in view of Corollary \ref{maxhq}, we conclude that $g$ is a maximal curve of degree $k-1$ in $GC_{n-1}$-set $\Xset_1:=\Xset\setminus \ell_1.$  Then, by induction hypothesis, we have that
$$g=\ell_2\cdots\ell_k,$$
where the set $\ell_i\setminus (\ell_2\cup\cdots\cup\ell_{i-1})$ contains exactly $n+2-i$ nodes of $\Xset_1,\ i=2,\ldots, k,$ i.e.,  the set $\ell_i\setminus (\ell_1\cup\cdots\cup\ell_{i-1})$ contains exactly $n+2-i$ nodes of $\Xset.$

$(\Leftarrow)$
Now assume that \eqref{mmm} takes place. Let us prove that $f$ is a maximal curve.
In view of Proposition \ref{maxcor} it suffices to prove that each node $A\in \Xset\setminus f$ uses $f.$ 

Let us use induction by $k$ again. The case $k=1$ is trivial.
Now by induction hypothesis we have that $\ell_1\cdots\ell_{k-1}$ is a maximal curve and hence $A$ uses it:
$$P^\star_A=\ell_1\cdots\ell_{k-1}g,\ g\in\Pi_{n-k+1}.$$
Then we have that the set $\ell_k\setminus (\ell_1\cup\cdots\cup\ell_{k-1})$ contains exactly $n+2-k$ nodes of $\Xset.$ Note that at these nodes $P^\star_A$ vanishes while $\ell_1, \ldots\ell_{k-1}$ do not vanish. Therefore at these $n-k+2$ nodes vanishes the polynomial $g\in\Pi_{n-k+1}.$
Now, by using Proposition \ref{prp:n+1ell}, we obtain that $g=\ell_k r,$ where $r\in\Pi_{n-k}.$ Hence 
$$P^\star_A=\ell_1\cdots\ell_k r,$$
that is, $A$ uses $f.$\end{proof}

\begin{definition} A finite set $\Lset$ of lines is said to be in \emph{general position,} if 
\begin{enumerate}
\item
no two lines of $\Lset$ are parallel, and 
\item
no three lines of $\Lset$ are concurrent.
\end{enumerate}
\end{definition}

Let a set $\Lset=(\ell_1,\ldots,\ell_{n+2})$ of lines be in general position.
Then the set $\Xset$ of ${n+2\choose 2}$ intersection points
of these lines is called Chung-Yao set of degree $n$ (see \cite{CY77}). Note that each fixed node belongs to $2$ lines, and the product of remaining $n$ lines gives the fundamental
polynomial of the fixed node. Thus $\Xset$ is $GC_n$-set. Notice that the lines $\ell_i,\ i=1,\ldots,n+2,$ are maximal for $\Xset.$ In view of Proposition \ref{hatk}(iii) no other line is maximal for the set $\Xset,$ that is, $\Lset=\Mset(\Xset).$

\begin{corollary} 
Suppose that $\Xset$ is Chung-Yao lattice of degree $n$ with set of maximal lines $\Lset.$ Suppose also that $f$ is a maximal curve of degree $k.$ Then $f$ is a product of $k$ maximal lines from $\Lset.$ 

Furthermore, if $f_1$ and $f_2$ are maximal curves of degrees $k_1$ and $k_2,$ respectively, with no common components, then we have that $k_1+k_2\le n+2.$
\end{corollary}
Indeed, in this case all the lines on the right-hand side of the equality \eqref{n1nk0} are maximal lines, hence all the lines on the right-hand side of the equality \eqref{n1nk} are maximal lines too.

At the end we consider constructions of $n$-correct sets with maximal curves, which are not products of lines. 

Note that, in view of Propositions \ref{HT1} and \ref{HT2}, any algebraic curve $f$ of degree $k\ge 1$ with no multiple components is maximal curve of degree $k$ for an $n$-correct set $\Xset,$ where $n\ge k.$

Let us then consider two arbitrary algebraic curves with no multiple components: $f$ and $g$ of degrees $m$ and $k,$ respectively, which intersect at $mk$ distinct points:
$$
\Iset:=\Iset(f,g)=f\cap g, \ \#\Iset=mk. 
$$

Below, for each value of $\delta=0, 1,$ we provide a construction of an $n$-correct set $\Xset,\ n=m+k-2+\delta,$ for which both curves $f$ and $g$ are maximal curves.

For this end consider  an $(m-2+\delta)$-correct set in $f\setminus\Iset$ and
an $(k-2+\delta)$-correct set in $g\setminus\Iset,$ denoted by $\Cset(f)$ and $\Cset(g),$ respectively. This can be done by using the Berzolari-Radon construction of $n$-correct sets (see \cite{HJZ09b}, Sec. 2.1) as well as the fact  that one can find lines that intersect $f$ and $g$ in $m$ and $k$ distinct points, respectively (see \cite{W} Thm.2.2). 

For the set of nodes $$\Xset:=\Xset(\delta):=\Iset(f,g)\cup\Cset(f)\cup\Cset(g)$$ we have the following proposition (for the first statement see \cite{H04}).
\begin{proposition} Let $\delta =0$ or $1.$ Then  $\Xset:=\Xset(\delta)$ is an $n$-correct set, where $n=m+k-2+\delta.$ Moreover, $f$ and $g$ are maximal curves of degrees $m$ and $k,$ respectively.
\end{proposition} 
\begin{proof} First let us check that $\#\Xset={n+2\choose 2}.$ Indeed, we have that
$$\#\Xset=\#\Iset(f,g)+\#\Cset(f)+\#\Cset(g)=mk+{m+\delta\choose 2}+{k+\delta\choose 2}={m+k+\delta\choose 2}.$$

Then let us check that $\Xset$ is $n$-correct set. To this end, let us assume, given Proposition \ref{correctii}(ii), that $p\in\Pi_n,\ p|_\Xset=0$ and prove that $p=0$.

Indeed, by using the Max Noether fundamental theorem we obtain that
$$p=Af+Bg,\ \hbox{where}\ A\in\Pi_{k-2+\delta},\ B\in\Pi_{m-2+\delta}.$$
From here we readily get that $A|_{\Cset(g)}=0$ and $B|_{\Cset(f)}=0.$ Since the sets $\Cset(f)$ and $\Cset(g)$ are respectively $(m-2+\delta)$ and $(k-2+\delta)$-correct, we obtain that $A=B=0$ and therefore $p=0.$

Finally, we have that 
$\#(f\cap\Xset)=\#\Iset + N_{m-2+\delta}=d(n,m),$ and similarly $\#(g\cap\Xset)=\#\Iset + N_{k-2+\delta}=d(n,k).$ Therefore $f$ and $g$ are maximal curves of degrees $m$ and $k,$ respectively. 
\end{proof}

The work of G. Vardanyan was supported by the Higher Education and Science Committee of RA, in the frames of the research project 21AG-1A045.

The work of N. Vardanyan was supported by the Higher Education and Science Committee of RA, in the frames of the basic financing
in YSU.

\end{document}